\documentclass[final]{siamltex}

\usepackage{amssymb}
\usepackage{amsfonts}
\usepackage{amsmath}
\usepackage{algorithm}
\usepackage{algorithmic}

\newcommand{\R}{{\mathbb R}}
\newcommand{\C}{{\mathbb C}}

\title{Vandermonde factorizations of a regular Hankel matrix and their application on the computation
 of B\'ezier curves}

\author{Licio H. Bezerra\thanks{Departamento de Matem\'atica,
        Universidade Federal de Santa Catarina,
        Florian\'opolis, SC,
        Brazil 88040-900 ({\tt licio@mtm.ufsc.br}).}}

\begin{document}

\maketitle

\begin{abstract}
 In this paper, a new method to compute a B\'ezier curve of degree $n=2m-1$ is introduced, here formulated as a Bernstein-Hankel form in $\C^m$, that is, each coordinate of the curve is of the form $e_m^T B^e_m(s) H B^e_m(s)^T e_m$, where $B^e_m(s)$ is a $m\times m$ lower triangular Bernstein matrix and $H$ is a Hankel matrix. The method depends on Vandermonde factorizations of a regular Hankel matrix, and so we begin with a proof, which utilizes Pascal matrices techniques, that given a regular Hankel matrix $H$, there is a finite set of complex numbers $\gamma$ such that $x^m - p_{m-1}x^{m-1} - ... - p_0$ has multiple roots, where $(p_0\, ...\, p_{m-1}) = (h_{m+1}\, ... \, h_n \, \gamma)\, H^{-1}$. Therefore, a Vandermonde factorization of $H$ can be accomplished by taking a complex number at random, and the Bernstein-Hankel form can be easily calculated, thus yielding points on the B\'ezier curve. We also see that even when $H$ is nearly singular, the method still works by shifting the skew-diagonal of $H$. By comparing this new method with a Pascal matrix method and Casteljau's, we see that the results suggest that this new method is very effective with regard to accuracy and time of computation for various values of $n$.
\end{abstract}



\pagestyle{myheadings}
\thispagestyle{plain}
\markboth{L. H. Bezerra}{On the computation of B\'ezier curves}

\section{Introduction}

Let $H$ be a Hankel matrix of order $n$, i.e., $(\forall i,j \in \{ 1,...,n\} )$ $H_{ij}=h_{i+j-1}$. A very known theorem says that, if $H$ is nonsingular, then
a Vandermonde matrix $V$ and a diagonal matrix $D$ exist such that
$H=VDV^T$. There is a proof of this fact in \cite{heinig_rost}, which
utilizes a class of matrices arisen in the theory of root separation of algebraic polynomials,
namely the class of Bezoutians.
Here, in section \$ 2, from a procedure that is currently utilized in linear prediction to estimate parameters in exponential modeling, it is showed that the spectrum of the companion matrix $C=C(x_{\gamma})$, where $x_{\gamma}$ is the solution of the linear prediction system $Hx=y_{\gamma}$, with $y_{\gamma}=(h_{n+1}\, ...\, h_{2n-1}\, \gamma)^T$, is simple for all but a finite set of $\gamma$. For the values belonging to this finite set, there is
a more general factorization: $H = V_c D V_c^T$, where $V_c$ is a confluent Vandermonde matrix and $D$ is a block diagonal matrix, as it can be seen in \cite{boley2}.
Our approach to the proof of the Vandermonde factorization of a nonsingular Hankel matrix is very similar to the one found in \cite{fiedler}, but the proofs are distinct. For instance, we make here use of generalized Pascal matrices to quickly obtain some general properties of Hankel matrices.

In section \$3, we see that a B\'ezier curve of degree $n-1$, where $n=2m-1$, can be described as a Bernstein-Hankel form on $\C^m$. Also, in this section a new algorithm to compute B\'ezier curves is proposed, from a Vandermonde factorization of the associated Hankel matrix. In section \$ 4, results of numerical experiments are presented, which strongly suggest that we can compute those curves in a very fast and precise way. That is corroborated from the comparisons done with the Casteljau's method (\cite{casteljau}) with various values of $n$. On the other hand, however, several experiments indicate that the computation of Vandermonde factorization of a Hankel matrix is sensitive to its condition with respect to inversion. However, once its skew-diagonal entries are shifted toward skew-diagonal dominance the precision of the computation improves, which is a simple and efficient way to deal with the instability of Vandermonde factorization of ill-conditioned Hankel matrices, at least for the computation of B\'ezier curves from this approach.

\section{Vandermonde factorizations of a nonsingular Hankel matrix}

\

Let $H = \left(\begin{array}{cccc} h_1 & h_2 & ... & h_n\\
\vdots & \vdots & \vdots & \vdots \\
h_{n-1} & h_n & ... & h_{2n-2} \\
h_n & h_{n+1} & ... & h_{2n-1} \end{array} \right)$. Suppose $H$ is nonsingular.
Let $x_{\gamma}$ be the solution of the linear prediction system $Hx=y_{\gamma}$, where $y_{\gamma}=(h_{n+1}\, ...\, h_{2n-1}\, \gamma)^T$.
We want to show that the set of $\gamma \in \C$ for which the companion matrix
$C_{\gamma}=compan(x_{\gamma})$ is not diagonalizable is finite.
Since $C_{\gamma}$ is a nonderogatory matrix, it suffices to show that $S$, the set of scalars $\gamma$ such that the spectrum of $C_{\gamma}$ is not simple, is finite.
This means that, out of this set, the characteristic polynomial of $C_{\gamma}$, $p_{\gamma}(x)$, doesn't have multiple roots.
If $a=(a_0 \, ... \, a_{n-1})^T$ and $b=(b_0 \, ... \, b_{n-1})^T$ are the respective solutions of
$Hx = e_n= (0\, ...\, 0\, 1)^T$ and $Hx = (h_{n+1}\, ...\, h_{2n-1} \, 0)^T$, then
$p_{\gamma}(x) = r(x) -\gamma s(x)$, where
$r(x) =
x^n - b_{n-1}x^{n-1} - ... - b_1x - b_0$ and $s(x)=a_{n-1}x^{n-1} + ...+a_1x + a_0$.
It is not difficult to see that $S$ is finite iff $r(x)$ and $s(x)$ don't have any common root.

\

\begin{lemma}
\label{nonzero}
Let $H$ be a $n\times n$ nonsingular Hankel matrix. If
$a=(a_0  ...  a_{n-1})^T$ and $b=(b_0 ...  b_{n-1})^T$ are the respective solutions of
$Hx = e_n$ and $Hx = (h_{n+1}\, ...\, h_{2n-1} \, 0)^T$,
then $a_0\ne 0$ or $b_0\ne 0$.
\end{lemma}

\begin{proof}

Suppose $\left| H(1:n-1,2:n)\right| \ne 0$. Therefore, from Cramer's rule, $a_0\ne 0$.
Let $x_1,...,x_{n-1}$ be the unique scalars such that
$$x_1\left( \begin{array}{c} h_2\\ \vdots \\ h_n\end{array} \right) + ... +
x_{n-1}\left( \begin{array}{c} h_n\\ \vdots \\ h_{2n-2}\end{array} \right)=
\left( \begin{array}{c} h_{n+1}\\ \vdots \\ h_{2n-1}\end{array} \right).
$$
Hence,
$x=(x_0\, x_1\, ...\, x_{n-1})^T=\gamma a + b$ is the solution of $Hx=(h_{n+1}\, ... \, h_{2n-1} \, \gamma)^T$, with $x_0=0$, iff
$\gamma = x_1 h_{n+1} + ... + x_{n-1} h_{2n-1}$. For other complex numbers $\gamma$, $x_0=\gamma a_0 + b_0 \ne 0$, that is, $a_0\ne 0$ or $b_0\ne 0$. Notice that
$a_0\ne 0$, and $b_0=0$ iff $x_1 h_{n+1} + ... + x_{n-1} h_{2n-1}=0$.

Now, suppose $H(1:n-1,2:n)=H(2:n,1:n-1)$ is singular. First, since $H$ is nonsingular, the dimension of $span\{ H(2:n,1),...,H(2:n,n-1), H(2:n,n) \}$ is $(n-1)$, as well as the dimension of $span\{ H(1:n-1,1),...,H(1:n-1,n-1), H(1:n-1,n) \}$. Hence, $H(2:n,n) \notin
span\{ H(2:n,1),...,H(2:n,n-1)\}$, whose dimension is $n-2$. On the other side, $H(2:n,n) \in span\{ H(1:n-1,1),...,H(1:n-1,n)\}=span\{ H(1:n-1,1),H(2:n,1),...,H(2:n,n-1)\}$, and so,
there exist $x_0,...,x_{n-1}$, where $x_0$ is different from zero and unique, such that
$$
\left( \begin{array}{c} h_{n+1}\\ \vdots \\ h_{2n-1}\end{array} \right)=
x_0\left( \begin{array}{c} h_1\\ \vdots \\ h_{n-1}\end{array} \right) +
x_1\left( \begin{array}{c} h_2\\ \vdots \\ h_n\end{array} \right) +
... +
x_{n-1}\left( \begin{array}{c} h_n\\ \vdots \\ h_{2n-2}\end{array} \right).
$$
Observe that, in this case, for all $\gamma \in \C$, $x_0 = b_0\ne 0$, and $a_0=0$.
\qquad\end{proof}

\

From the above proof, there can be at most one complex number $\gamma$ such that
$p_{\gamma}(0)=-b_0 - \gamma a_0=0$
We can also conclude from the lemma \ref{nonzero} that zero is not a common root of $r(x) =
x^n - b_{n-1}x^{n-1} - ... - b_1x - b_0$ and $s(x)=a_{n-1}x^{n-1} + ...+a_1x + a_0$.

\

Now, define $H_{\gamma}^{\kappa} =
\left(\begin{array}{cccc} h_1 & ... & h_n & h_{n+1}\\
\vdots & \ddots & \vdots & \vdots \\
h_n & ... & h_{2n-1} & \gamma \\
h_{n+1}& ... &  \gamma & \kappa  \end{array} \right)$.
Since $H$ is nonsingular, $H_{\gamma}^{\kappa}$ is also nonsingular iff
$\kappa \ne \kappa_0=(h_{n+1}\,\, ... \,\, h_{2n-1} \,\, \gamma) \, \, H^{-1} (h_{n+1}\,\, ... \,\, h_{2n-1} \,\, \gamma)^T$,
which is equal to $(h_{n+1}\, ... \, h_{2n-1} \, \gamma) (b_0 + \gamma a_0 \,\,\, ... \,\,\, b_{n-2} + \gamma a_{n-2} \,\,\, b_{n-1} + \gamma a_{n-1})^T$.

\

Note that $H_{\gamma}^{\kappa} \left( \begin{array}{c} -b_0 - \gamma a_0 \\ \vdots \\ -b_{n-1} - \gamma a_{n-1} \\ 1\end{array} \right)=(\kappa - \kappa_0)
\left( \begin{array}{c} 0 \\ \vdots \\ 0 \\ 1\end{array} \right)$.
%
%
Therefore, lemma \ref{nonzero} can be rewritten
as the following lemma:

\begin{lemma}
\label{nonnonzero}
Let $H_{\gamma}^{\kappa} =
\left(\begin{array}{cccc} h_1 & ... & h_n & h_{n+1}\\
\vdots & \ddots & \vdots & \vdots \\
h_n & ... & h_{2n-1} & \gamma \\
h_{n+1}& ... &  \gamma & \kappa  \end{array} \right)$ be a Hankel matrix, where
$H = H_{\gamma}^{\kappa} (1:n,1:n)$
is nonsingular. Suppose that $H_{\gamma}^{\kappa}$ is also nonsingular, that is,
$\kappa \ne (h_{n+1}\,\, ... \,\, h_{2n-1} \,\, \gamma) \, \, H^{-1} (h_{n+1}\,\, ... \,\, h_{2n-1} \,\, \gamma)^T$. Let $p$ be the solution of $H_{\gamma}^{\kappa}x = e_{n+1}$. Then, except for one possible complex number $\gamma$,
$p_0\ne 0$.
\end{lemma}

\

Now, let $\alpha$ be any complex number and $q_{\gamma}(x) = p_{\gamma}(x+\alpha) = r(x+\alpha) - \gamma s(x+\alpha)$.
In an analogous way to the proof for $\alpha=0$, it will be shown that $r(\alpha)$ and $s(\alpha)$ cannot be both null because there can be only one complex number $\gamma$ such that $q_{\gamma}(0)=0$.
To prove this, we introduce some notations and definitions in the following.

\

\begin{definition}
Let $\alpha \in \C$.
$P_n[\alpha]$ be the $n\times n$ is the
lower triangular matrix defined for each $i,j\in \{  1, 2, \ldots, n\}$ by
$$\left( P_n[\alpha]\right)_{ij}= \left\{ \begin{array}{ccl}
\alpha^{i-j} \binom{i-1}{j-1} &,& \mbox{for } i \geqslant j; \\
0&,& \mbox{otherwise}.
\end{array} \right.
$$
$P_n[\alpha]$ is said a generalized lower triangular Pascal matrix.
If $\alpha = 1$,
$P_n[1]=P_n$ is called the $n\times n$ lower triangular Pascal matrix.
\end{definition}

\

\noindent Some results about these matrices (see \cite{cav}, \cite{aceto}) are listed in the following lemma:

\

\begin{lemma} Let $P_n[\alpha]$ a generalized lower triangular Pascal matrix. Then,
\label{facts}
\begin{enumerate}
\item[(a)] $P_n[0]=I_n$;
\item[(b)] $P_n[\alpha] P_n[\beta]=P_n[\alpha + \beta]$;
\item[(c)] $(P_n[\alpha])^{-1} = P_n[-\alpha]$;
\item[(d)] Let $\alpha\ne 0$ and let $G_n(\alpha)$ be the $n\times n$ diagonal matrix such that, for all
$k\in \{ 1,...,n\}$, $\left( G_n(\alpha)\right)_{kk} = \alpha^{k-1}$. Then $P_n[\alpha] = G_n(\alpha)P_n G_n(\alpha)^{-1}=G_n(\alpha)P_n G_n(\alpha^{-1})$. In particular, $P_n^{-1} = G_n(-1) P_n G_n(-1)$.
\end{enumerate}
\end{lemma}

\

\begin{definition}
For $s \in [0,1]$, the $n\times n$ Bernstein matrix $B_n^e(s)$ is the matrix defined
for each $i,j\in \{  1, 2, \ldots, n\}$ as follows:
$$
[B_n^e(s)]_{ij} = \left\{ \begin{array}{rcl}
\binom{i-1}{j-1}s^{j-1}(1-s)^{i-j}
&,& \mbox{for } i\ge j; \\
0&, & \mbox{otherwise}.
\end{array}
\right.
$$
\end{definition}
A very important fact about Bernstein matrices, which will be used here later, is the following proposition, whose proof can be found in \cite{aceto}:

\begin{proposition}
\label{bernpasc}
Let $s \in [0,1]$ and let $B_e(s)$ be a $n\times n$ Bernstein matrix Then,
$B_n^e(s) = P_n G_n(s) P_n^{-1},$
where $P_n$ is the $n\times n$ lower triangular Pascal matrix and $G_n(s)=diag([1,s,...,s^{n-1}])$.
\end{proposition}

\

In the following, we present some relations between Pascal and Hankel matrices.

\

\begin{lemma}
\label{hankpascal}
Let $H$ be a $n\times n$ Hankel matrix and let $P_n$ be the $n\times n$ lower triangular Pascal matrix.
Then $P_n H P_n^T$ is still a Hankel matrix.
\end{lemma}
\begin{proof}
The lemma obviously holds when $n=1$. Suppose it holds for all Hankel matrices $H$ of order $n\ge 1$.
Now, let $H$ be a $(n+1)\times (n+1)$ Hankel matrix and consider $P_{n+1} H P_{n+1}^T$.
Since $P_{n+1} H P_{n+1}^T$ is symmetric and
$P_{n+1} H P_{n+1}^T = [P_{n} H P_{n}^T\, v; v^T \, \kappa]$, for some $v\in  \C^n$, by induction it suffices to show that, for all
$k\in\{ 1,...,n-1\}$, $(P_{n+1} H P_{n+1}^T)_{n+1,k}= (P_{n+1} H P_{n+1}^T)_{n,k+1}$.
Now,
$$
(P_{n+1} H P_{n+1}^T)_{n+1,k}= e_{n+1}^T P_{n+1} \sum_{j=0}^{k-1}  \binom{k-1}{j} H e_{j+1}=
$$
$$
=\sum_{i=0}^{n}\sum_{j=0}^{k-1} \binom{n}{i} \binom{k-1}{j} e_{i+1}^TH e_{j+1}=
\sum_{s=2}^{n+k-1} h_{s-1} \sum_{i=0}^{s} \binom{n}{i}\binom{k-1}{s-i},
$$
which is equal, from Vandermonde convolution (\cite{knuth}), to
$$
\sum_{s=2}^{n+k-1} h_{s-1} \sum_{i=0}^{s} \binom{n-1}{i}\binom{k}{s-i}=
\sum_{i=0}^{n-1}\sum_{j=0}^{k} \binom{n-1}{i} \binom{k}{j} e_{i+1}^TH e_{j+1}=
$$
$$
=e_{n}^T P_{n+1} \sum_{j=0}^{k}  \binom{k}{j} H e_{j+1}=(P_{n+1} H P_{n+1}^T)_{n,k+1}.
$$
\qquad\end{proof}

\

\begin{corollary}
\label{diag}
Let $H$ be a $n\times n$ Hankel matrix and $\alpha$ be a complex number.
Then, $P_n[\alpha]H P_n[\alpha]^T$ is still a Hankel matrix.
\end{corollary}
\begin{proof}
For $\alpha=0$, the result follows from lemma \ref{hankpascal}. Let $\alpha \ne 0$. Since from lemma \ref{facts} $P_n[\alpha] = G_n(\alpha)P_n G_n(\alpha^{-1})$,
where $G_n(\alpha)= diag\, (1,\alpha,...,\alpha^{n-1})$,
it suffices to show that $G(\alpha) H G(\alpha)$ is a Hankel matrix. But this is obviously true, for $\left( G(\alpha) H G(\alpha) \right)_{ij} = h_{i+j-1} \alpha^{i+j-2}$.
\qquad\end{proof}

\

Next we give a proof that $r(x)$ and $s(x)$ don't have any common root by using a generalized Pascal matrix technique.

\begin{proposition}
\label{pol}
Let $H$ be a $n\times n$ nonsingular Hankel matrix.
Let $a=(a_0 \, a_1 \, ... \, a_{n-1})^T$ and $b=(b_0 \, b_1 \, ... \, b_{n-1})^T$ be the solutions of $Ha=e_n$ and $Hb = (h_{n+1}\, ... \, h_{2n-1} \, 0)^T$, respectively.
Then $r(x) = x^n - b_{n-1}x^{n-1} - ... - b_1x - b_0$ and $s(x)=a_{n-1}x^{n-1} + ...+a_1x + a_0$ don't have
any common root.
\end{proposition}
\begin{proof} Let $\gamma \in \C$ and let $p_{\gamma}=\left( -b_0-\gamma a_0\, \, ... \, -b_{n-1}-\gamma a_{n-1}\, \, 1\right)^T$. Let $q_{\gamma}=(q_0 \, ... \, q_{n-1} \, 1)^T$ be the vector of coefficients of the polynomial $r(x+\alpha)-\gamma s(x+\alpha)$.
We note that
$q_{\gamma} = P_{n+1}[\alpha]^T p_{\gamma} =  P_{n+1}[\alpha]^T (H_{\gamma}^{\kappa})^{-1} e_{n+1}$, for $\kappa = 1+\kappa_0$. Thus, $H_{\gamma}^{\kappa} P_{n+1}[\alpha]^{-T} q_{\gamma} = e_{n+1}$, and so,
$$\widehat H_{\gamma}^{\kappa}q_{\gamma} = P_{n+1}[\alpha]^{-1} H_{\gamma}^{\kappa} P_{n+1}[\alpha]^{-T} q_{\gamma} = P_{n+1}[-\alpha] H_{\gamma}^{\kappa} P_{n+1}[-\alpha]^T q_{\gamma}=e_{n+1}.$$
$\widehat H_{\gamma}^{\kappa}$ is also nonsingular and, from corollary \ref{diag}, is a Hankel matrix.
Since $H_{\gamma}^{\kappa}= H_0^0 + \gamma \, ( e_{n+1} e_n^T + e_n e_{n+1}^T) + \kappa \, e_{n+1}e_{n+1}^T$,
we see that
$\widehat H_{\gamma}^{\kappa} = \widehat H_{0}^{0} + \gamma \, ( e_{n+1} e_n^T + e_n e_{n+1}^T) +
(\kappa -2\,n\, \alpha) \, e_{n+1}e_{n+1}^T$. That is,
$$\widehat H_{\gamma}^{\kappa} =
\left(\begin{array}{cccc} \hat h_1 & ... & \hat h_n & \hat h_{n+1}\\
\vdots & \ddots & \vdots & \vdots \\
\hat h_n & ... & \hat h_{2n-1} & \hat \gamma \\
\hat h_{n+1}& ... &  \hat \gamma & \hat \kappa  \end{array} \right),$$
where $\widehat H_{\gamma}^{\kappa}(1:n,1:n) = \widehat H =P_n[-\alpha] H P_n[-\alpha]^T$ is nonsingular and $\hat \gamma = \gamma + (\widehat H_{0}^{0})_{n+1,n}$.
Thus, from lemma \ref{nonnonzero}, except for one possible complex number $\gamma$,
$(q_{\gamma})_0\ne 0$.
\qquad\end{proof}

\

Note that $(q_{\gamma})_0= 0$ only when $s(\alpha)\ne 0$, that is,
when $\left| \widehat H_{\gamma}^{\kappa}(1:n-1,2:n)\right| = \left| \widehat H(1:n-1,2:n)\right|\ne 0$.
In this case, $\gamma = r(\alpha)/s(\alpha)$.

\

\begin{proposition}
\label{theprop}
Let $\gamma \in \C$.
Let $p_{\gamma}(x) = x^n - b_{n-1}x^{n-1} - ... - b_0 - \gamma (a_{n-1}x^{n-1} + ... + a_0)=r(x) - \gamma s(x)$
the characteristic polynomial of $C_{\gamma} = H_1(\gamma)H^{-1}$, where
$H_1(\gamma)$ is the Hankel matrix defined by $H_1(\gamma)e_k = He_{k+1}$ for $k=1,...,n-1$
and $H_1(\gamma) e_n = (h_{n+1}\, ... \, h_{2n-1} \, \gamma)^T$, that is, $C_{\gamma} = [e_2^T; ...; e_n^T;
(h_{n+1}\, ... \, h_{2n-1} \, \gamma)H^{-1}]$. Then the set of scalars $\gamma$
such that $C_{\gamma}$ is not diagonalizable is finite.
\end{proposition}

\begin{proof}
$C_{\gamma}$ is a companion matrix, and hence, a nonderogatory matrix. Thus, it suffices to show that the set of scalars $\gamma$ such that the spectrum of $C_{\gamma}$ is not simple is finite.

Let $\alpha\in \C$ be an eigenvalue of $C_{\gamma}$, that is, a root of $p_{\gamma}(x)$.
Therefore, $r(\alpha) = \gamma s(\alpha)$.
Then, from proposition \ref{pol}, $s(\alpha)\ne 0$.
So, there are two cases:
\begin{romannum}
\item $r(\alpha)=0$, and this occurs iff $\gamma=0$. In this case,
$C_{0}$ is not diagonalizable iff $r'(\alpha)=0$.
\item $r(\alpha)\ne 0$, which means that $\gamma = r(\alpha)/s(\alpha)$. Therefore,
$p_{\gamma}'(\alpha)=0$ iff $r'(\alpha)=s'(\alpha)=0$, or $s'(\alpha)\ne 0$ and $r'(\alpha)=\gamma s'(\alpha)$.
\end{romannum}
Therefore, since $s\ne 0$ and $r/s$ is not a constant,
$\alpha$ is contained in the set of the roots of $r's-rs'$, which has at most $2(n-1)$ elements.
Hence, we can conclude that $\{ \gamma \in \C\, | \, C_{\gamma} \mbox { is not diagonalizable}\}$
is finite and has at most $2(n-1)$ elements.
\qquad\end{proof}

\

We can now state the following theorem:

\

\begin{theorem}
\label{thethe}
Let $H$ be a $n\times n$ nonsingular Hankel matrix.
Let $r(x) = x^n - b_{n-1}x^{n-1} - ... - b_0$ and $s(x)=a_{n-1}x^{n-1} + ... + a_0$, where
$a=(a_0 \, a_1 \, ... \, a_{n-1})^T$ and $b=(b_0 \, b_1 \, ... \, b_{n-1})^T$ are such that $Ha=e_n$ and $Hb = (h_{n+1}\, ... \, h_{2n-1} \, 0)^T$. Let $S = \{\, \alpha \in \C \, | \, (rs'-r's)(\alpha)=0 \mbox{ and } s(\alpha)\ne 0\, \}$ and $T = \{ r(\alpha)/s(\alpha)\, | \, \alpha \in S\}$. Then, for all $\gamma \in \C -T$,
$H = V_{\gamma} D_{\gamma} V_{\gamma}^T$, where $V_{\gamma} = vander (\alpha_1,...,\alpha_n)$, $D_{\gamma} = diag\, (V_{\gamma}^{-1}He_1)$, $\{\, \alpha_1,...,\alpha_n\, \} = \lambda(C_{\gamma})$, and
$C_{\gamma}$ is the companion matrix whose last row is
$(b_0 + \gamma a_0\, ... \, b_{n-1}+\gamma a_{n-1})$.
\end{theorem}
\begin{proof}
From proposition \ref{theprop}, for all $\gamma \in \C -T$, $\lambda(C_{\gamma})$ is simple. Suppose
$\{\, \alpha_1,...,\alpha_n\, \} = \lambda(C_{\gamma})$. Let $v=(h_{n+1}\, ... \, h_{2n-1} \, \gamma)^T$ and $H_1 = [H(2:n,:);v]$.
Then,
$$C_{\gamma}=H_1H^{-1}=V_{\gamma}\, diag([\alpha_1,...,\alpha_n])\, V_{\gamma}^{-1},$$
where $V_{\gamma}e_i = \left( 1 \, \alpha_i \, ... \, \alpha_i^{n-1}\right)^T$, for all $i\in \{ 1,...,n\}$.
So,
$$V_{\gamma}^{-1}H_1=  \, diag([\alpha_1,...,\alpha_n])\, V_{\gamma}^{-1} H.$$
Let $d =(d_1\, ... \, d_n)^T = V_{\gamma}^{-1}He_1$.
Hence, for all $i\in \{ 1,...,n\}$,
$$V_{\gamma}^{-1}He_i = diag([\alpha_1,...,\alpha_n])^{i-1} d=\left( d_1\alpha_1^{i-1}\, ... \, d_n \alpha_n^{i-1} \right)^T= D_{\gamma} V_{\gamma}^T e_i.$$
\qquad\end{proof}

\section{B\'ezier curve as a Hankel form}

Efficient methods to compute B\'ezier curves of degree $n-1$ (\cite{bezier}) are fundamentals tools in Computed-Aided Geometric Design area. The Casteljau's algorithm is a widespread method for this computation. However, for each $s\in (0,1)$ it demands ${\cal O}(n^2)$ multiplications. For $n$ not very large, there are more efficient methods, like the ones introduced in \cite{phien}, where the computation of points on these curves is carried out by generalized Ball curves, or the ones presented in \cite{licio-sacht}, which use fast Pascal matrix-multiplication.
Here we show that we can describe a B\'ezier curve as a Hankel form and, hence, we see that we can easily compute points of the curve from a Vandermonde factorization of the associated Hankel matrix.

Let $Q_0 = (x_0,y_0)$, $Q_1=(x_1,y_1)$, ..., $Q_{n-1}=(x_{n-1},y_{n-1})$
be $n$ points in $\R^2$.
B\'ezier has his name on the curve $B$ defined from these $n$ points
as follows:
$$
B(s) = \left( \begin{array}{c} b_1(s)\\ b_2(s) \end{array} \right)= \sum_{i=0}^{n-1} \binom{n-1}{i} s^i (1-s)^{n-1-i}Q_i, \hspace{1cm} s \in [0,1].
$$
Let $x= (x_0 \, ... \, x_{n-1})^T$ and $x= (y_0 \, ... \, y_{n-1})^T$.
Then, for each $s\in [0,1]$,
$$
b_1(s) = e_n^T B_n^e(s) x \mbox{ and } b_2(s) = e_n^T B_n^e(s) y,
$$
where $B_n^e(s)$ is a $n\times n$ Bernstein matrix. Thus, from lemma \ref{facts}, for each $s\in [0,1]$,
\begin{equation}
\label{pascalbezier}
b_1(s) = e_n^T P_n G_n(s) P_n^{-1} x \mbox{ and } b_2(s) = e_n^T P_n G_n(s) P_n^{-1} y.
\end{equation}
In the following, we discuss different approaches that make use of (\ref{pascalbezier}) to compute a B\'ezier curve.

\

We can notice that, if
$B(s) = B_{Q_0Q_1\ldots Q_{n-1}}(s)$ denotes the B\'ezier curve determined by the points $Q_0$, $Q_1$, ...,
$Q_{n-1}$, then
$$B(s) = (1-s)B_{Q_0Q_1\ldots Q_{n-2}}(s) + sB_{Q_1Q_2\ldots Q_{n-1}}(s).$$
Without loss of generality, from now on we will suppose that $n$, the number of control points of a B\'ezier curve, is odd: $n=2m-1$, $m>1$. In this case,
it is easy to conclude by induction that, for all $k=0,...,m-1$,
$${B(s) = \sum_{j=0}^{k} \binom{k}{j} (1-s)^{k-j} s^j B_{Q_jQ_{j+1}\ldots Q_{j+n-k-1}}(s) }.$$
Particularly, for $k=m-1$ we have
$$
B(s) = \sum_{j=0}^{m-1} \binom{m-1}{j} (1-s)^{m-1-j} s^j B_{Q_jQ_{j+1}\ldots Q_{j+m-1}}(s),
$$
and so,
$$
b_1(s)= \sum_{j=0}^{m-1} \binom{m-1}{j} (1-s)^{m-1-j} s^j e_m^T P_mG(t) P_m^{-1}x_{j...j+m-1},$$
$$
b_2(s)= \sum_{j=0}^{m-1} \binom{m-1}{j} (1-s)^{m-1-j} s^j e_m^T P_mG(t) P_m^{-1}y_{j...j+m-1},
$$
where $x_{j...j+m-1}$ and $y_{j...j+m-1}$ denote the column vectors
$(x_j \ldots x_{j+m-1})^T$ and  $(y_j \ldots y_{j+m-1})^T$, respectively, for
$j=0,...,m-1$.
However,
$$
\sum_{j=0}^{m-1} \binom{m-1}{j} (1-s)^{m-1-j} s^j e_m^T P_mG(t) P_m^{-1}x_{j...j+m-1}=
$$
$$
=e_m^T P_mG(t) P_m^{-1}\left(\sum_{j=0}^{m-1} \binom{m-1}{j} (1-s)^{m-1-j} s^j x_{j...j+m-1}\right),
$$
and
$
\sum_{j=0}^{m-1} \binom{m-1}{j} (1-s)^{m-1-j} s^j x_{j...j+m-1}
$
is a column vector whose ith coordinate is
$e_m^TP_mG(t)P_m^{-1} x_{i-1...m+i-2}$.
Thus, we can state the following lemma, from which we can conclude that each coordinate of a B\'ezier curve is a Hankel form:

\

\begin{lemma}\label{lemab}
Let $n=2m-1$, where $m$ is an integer greater than 1 and
let $B(s) = \left( b_1(s) \, b_2(s) \right)^T$ be a B\'ezier curve of degree $n-1$
defined from  $n$ points
$Q_0 = (x_0,y_0)$, $Q_1=(x_1,y_1)$, ..., $Q_{n-1}=(x_{n-1},y_{n-1})$
in $\R^2$. Then
$$b_1(s) = e_m^T B_m^e(s) H_x (B_m^e(s))^Te_m \mbox{ and }
b_2(s) = e_m^T B_m^e(s) H_y (B_m^e(s))^Te_m,$$ where
$H_x=hankel(C_x,R_x)$ and $H_y=hankel(C_y,R_y)$ are $m\times m$ Hankel matrices whose first columns are $C_x = (x_0 ... x_{m-1})^T$ and $C_y = (y_0 ...y_{m-1})^T$  respectively, and whose last rows are $R_x = (x_{m-1},...,x_{n-1})$ and $R_y=( y_{m-1},...,y_{n-1}$ respectively.
\end{lemma}

\

\begin{corollary}
\label{impo}
Let $n=2m-1$, where $m$ is an integer greater than 1.
Let $B$ be a B\'ezier curve of degree $n-1$ defined from $n$ control points, and let
$x=(x_0 ... x_{n-1})^T$ and $y=(y_0 ... y_{n-1})^T$ be their respective vector of coordinates.
Let $H_x=hankel(C_x,R_x)$ and $H_y=hankel(C_y,R_y)$, where $C_x = (x_0 ... x_{m-1})^T$, $R_x = (x_{m-1},...,x_{n-1})$, $C_y = (y_0 ... y_{m-1})^T$ and $R_y = (y_{m-1},...,y_{n-1})$. If $H_x$ and $H_y$ are nonsingular,
then there exist complex numbers $d_1, ...,  d_n$, $t_1, ..., t_n$, $\hat d_1, ..., \hat d_n$ and
$\hat t_1, ..., \hat t_n$ such that
\begin{equation}
b_1(s) = \sum_{i=1}^m d_i (1-s+s. t_i)^{n-1} \mbox{ and } b_2(s) = \sum_{i=1}^m \hat d_i (1-s+s. \hat t_i)^{n-1}.
\label{theeq}
\end{equation}
\end{corollary}
\begin{proof}
If $H_x$ is nonsingular, from theorem \ref{thethe}, there is a Vandermonde matrix $V = vander([t_1,...,t_n])$ and a diagonal matrix $D=diag([d_1,...,d_n])$ such that $H_x = VDV^T$. So,
$$b_1(s) = e_m^T B_m^e(s) H_x (B_m^e(s))^Te_m = e_m^T B_m^e(s) VDV^T (B_m^e(s))^Te_m=$$
$$ =\sum_{i=1}^m d_i (1-s+s.t_i)^{2m-2} =  \sum_{i=1}^m d_i (1-s+s.t_i)^{n-1},$$
for $e_m^T B_m^e(s)Ve_i = \sum_{j=0}^{m-1} (1-s)^{m-1-j}.s^j.t_i^j= (1-s+s.t_i)^{m-1}$
for all $i\in \{1,...,m\}$. In an analogous way, we conclude that
$$ b_2(s) = \sum_{i=1}^m \hat d_i (1-s+s.\hat t_i)^{n-1},$$ for some $\hat d_1, ..., \hat d_n$ and
$\hat t_1, ..., \hat t_n$.
\qquad\end{proof}

The following proposition is about another representation of a B\'ezier curve of degree $n-1$.

\

\begin{proposition}\label{double}
Let $n=2m-1$, where $m$ is an integer greater than 1 and
let $B(s) = \left( b_1(s)\, b_2(s) \right)^T$ be a B\'ezier curve of degree $n-1$
defined from  $n$ points
$Q_0 = (x_0,y_0)$, $Q_1=(x_1,y_1)$, ..., $Q_{n-1}=(x_{n-1},y_{n-1})$
of $\R^2$. Then
$$b_1(s) = \sum_{k=0}^{n-1} a_k \binom{n-1}{k} s^k \mbox { and }
b_2(s) = \sum_{k=0}^{n-1} b_k \binom{n-1}{k} s^k, \mbox{ where }$$
$\left(\begin{array}{ccc} a_0 & ... & a_{n-1}\\
\vdots & \ddots  & \vdots \\
a_{m-1} &  ... & a_{n-1} \end{array} \right)=
P_m^{-1}H_xP_m^{-T}$ and $\left(\begin{array}{ccc} b_0 & ... & b_{n-1}\\
\vdots & \ddots  & \vdots \\
b_{m-1} &  ... & b_{n-1} \end{array} \right)=P_m^{-1}H_yP_m^{-T}$.
\end{proposition}
\begin{proof}
Let $A=P_m^{-1}H_xP_m^{-T}$ and $B=P_m^{-1}H_yP_m^{-T}$. From lemma~\ref{lemab}, it follows that
$$b_1(s) = e_m^T P_mG(s) A G(s) P_m^T e_m \mbox{ and }
b_2(s) = e_m^T P_mG(s)B G(s)P_m^T e_m.$$
Now, $e_m^TP_mG(s) = \left( \binom{m-1}{0}\, \binom{m-1}{1}s\, ...\, \binom{m-1}{m-1}s^{m-1}\right)$.
Therefore,
$$b_1(s) =  \sum_{k=0}^{n-1} a_k \left( \sum_{j=0}^k \binom{m-1}{j} \binom{m-1}{k-j} \right) s^k,$$
$$b_2(s) =  \sum_{k=0}^{n-1} b_k \left( \sum_{j=0}^k \binom{m-1}{j} \binom{m-1}{k-j} \right) s^k,$$
and the conclusion now follows from Vandermonde convolution
$$\sum_{j=0}^k \binom{m-1}{j} \binom{m-1}{k-j} = \binom{2(m-1)}{k}= \binom{n-1}{k}.$$
\qquad\end{proof}

\

We have just proved a property of the Pascal matrix-vector multiplication, which is remarked in the following corollary:

\begin{corollary}\label{gol}
Let $n=2m-1$, where $m\ge 1$, let $x=(x_0 ... x_{n-1})^T$
be a vector of $\C^n$ and let $H_x$ the Hankel matrix defined by $\left( H_x\right)_{ij}=x_{i+j-2}$. Then, $a=P_nx$, where $a=(a_0 ... a_{n-1})^T$ is such that $\left( P_m H_x P_m^T\right)_{ij}=a_{i+j-2}$.
\end{corollary}
\begin{proof} Let $y=G_n(-1)x$. Let $a = P_nx = G_n(-1)P_n^{-1}G_n(-1)x$. Then,
$$e_n^TP_nG_n(s)P_n^{-1}y=e_n^TP_nG_n(s)G_n(-1)a
= \sum_{k=0}^{n-1} (-1)^k a_k \binom{n-1}{k} s^k.$$
On the other hand, from proposition~\ref{double},
$(-1)^{i+j-2}a_{i+j-2}=\left( P_m^{-1} H_y P_m^{-T}\right)_{ij}$. Therefore,
$a_{i+j-2}=\left( P_m  H_x P_m^T\right)_{ij}$.
\qquad\end{proof}

\


\section{Numerical experiments}

We are going to compare B\'ezier curves of degree $(n-1)$ computed from the classical Casteljau's algorithm as well as from two other descriptions of the curve: as a Hankel form and by using the spectral decomposition $B^e_n (s) = P_nG_n(s)P_n^{-1}$.
We first observe that an uniform scaling of the control points of a B\'ezier curve yields an uniform scaling of the curve and if those points are translated by a vector $v=(p,q)$, then
the B\'ezier curve is also translated by $v$.
Hence, without loss of generality, we are going to assume that the coordinates of the control points are all real positive and also less than or equal to 1. So, we are going to use the MATLAB function $rand$ to generate $n$ test control points: $A=rand(n,2)$.

The Casteljau's algorithm is a very accurate algorithm to evaluate B\'ezier curves, for it is based on a numerically stable Bernstein matrix-vector multiplication:

\begin{algorithm}
\caption{Casteljau's algorithm}
\begin{algorithmic}
\STATE $n= length(x)$;
\STATE $x = x(:)$;
\STATE $ss = 1-s$;
\FOR{$k = 2:n$}
    \FOR{$t = n:-1:k$}
       \STATE x(t) = ss*x(t-1) + s*x(t);
    \ENDFOR
\ENDFOR
\STATE $b(s)=x(n)$
\end{algorithmic}
\end{algorithm}

This multiplication can be seen as a sequence of bi-diagonal matrix-vector multiplications, which becomes well explicit from the following lemma \cite{licio-sacht}:

\begin{lemma} Let $B_n^e(t)$ be a $n\times n$ Bernstein matrix.
Then
$$B_n^e(s)=E_{n-1}^e(s)...E_1^e(s) \mbox{ where, for $1\le k\le n-1$},$$
$E_k^e(s)=e_1e_1^T + ... + e_ke_k^T + e_{k+1}[(1-s)e_k+se_{k+1}]^T+...+e_n[(1-s)e_{n-1}+se_n]^T.$
\label{pasb}
\end{lemma}

\

Another way of calculating a B\'ezier curve is from its description as a Hankel form, which allows us to utilize a Vandermonde factorization of the associated Hankel matrix, and its algorithm is as follows:
\begin{algorithm}\nonumber
\caption{B\'ezier curve as a Hankel form}
\begin{itemize}
  \item given $n=2m-1$ distinct points of $\R^2$, let $H^x$ and $H^y$ be two $m\times m$ Hankel matrices formed from their coordinates;
  \item choose a number $\gamma$ at random and define the vectors $x_{\gamma}=(h^x_{m+1} \, ... \, h^x_n \, \gamma)^T$ and $y_{\gamma}=(h^y_{m+1} \, ... \, h^y_n \, \gamma)^T$;
  \item solve the systems $H^x z_{\gamma} = x_{\gamma}$ and $H^y w_{\gamma} = y_{\gamma}$ and consider the companion matrices $C_{z_{\gamma}}$ and $C_{w_{\gamma}}$;
  \item find the spectra of $C_{z_{\gamma}}$ and $C_{w_{\gamma}}$;
  \item define $d_x= (V^x_{\gamma})^{-1}H^x e_1$ and $d_y= (V^y_{\gamma})^{-1}H^y e_1$, where $V^x_{\gamma}$ and $V^y_{\gamma}$ are Vandermonde matrices formed from the spectrum of $C_{z_{\gamma}}$ and from the spectrum $C_{w_{\gamma}}$, respectively;
  \item for each $s\in [0,1]$, the B\'ezier curve $B(s)$ is then defined from the equation (\ref{theeq}).
  \end{itemize}
\end{algorithm}
We are supposing here that $H^x$ and $H^y$ are both nonsingular and that $\gamma$ is not one of those numbers which
yield in non-diagonalizable companion matrices.

\

The third way of computing a B\'ezier curve will be carried out by a
Pascal matrix method, which computes a B\'ezier curve $B(s)$ of degree $n-1$ via
the decomposition $B^e_n(s) = P_n G_n(-s) P_nG_n(-1)$:
\begin{algorithm}
\caption{Pascal matrix algorithm}
\begin{itemize}
\item given $n$, take $t\ge 1$ such that $P_n(t)$ is similar to
a lower triangular Toeplitz matrix $T=T(t)$ with maximum $\min T_{ij}/\max T{ij}$;
\item multiply $z=P_nG_n(-1)x = P_nx_-$  and $w=P_nG_n(-1)y = P_ny_-$ via fast Toeplitz matrix-vector
multiplication;
\item from a Horner-like scheme, evaluate the polynomials $e_n^T P_n G_n(-s)z$ and $e_n^T P_n G_n(-s)w$.
\end{itemize}
\end{algorithm}

We have used a fast Pascal matrix-vector
multiplication done from the similar Toeplitz matrix $T(t)$ (see \cite{wang}), 
where $t$ has been found by a procedure described in \cite{licio-sacht}, 
plus the $B(s)$ evaluation given by a Horner-like
scheme that evaluates the polynomial concomitantly with the binomial coefficients.
Since the $B(s)$-evaluation becomes unstable when $s$ approaches to 1,
we have introduced a simple procedure to improve the evaluation, that is
to divide the process of evaluation in two independent steps:
\begin{enumerate}
\item[(a)] evaluate $e_n^T P_n G_n(-s)z$ and $e_n^T P_n G_n(-s)w$ for $0\le s\le 1/2$;
\item[(b)] evaluate $e_n^T P_n G_n(-s)z_r$ and $e_n^T P_n G_n(-s)w_r$ for $1/2> s\ge 0$, which
is equivalent to evaluate $e_n^T P_n G_n(-s)z$ and $e_n^T P_n G_n(-s)w$ for $1/2 < s\le 1$.
\end{enumerate}

\begin{table}[ht]
\caption{Mean run time of computation of 129 points of a B\'ezier curve of degree N-1 by three
different methods: Casteljau's (C), Hankel form (H) and direct Pascal matrix method (P). The results of the second and third methods are compared to the ones obtained by Casteljau's via norm of the difference of the computed points
by the respective method and by Casteljau's.}

\begin{center} \footnotesize
\begin{tabular}{|c|c|c|c|c|c|} \hline
N & Time (Casteljau) &Time (Hankel) & Time (Pascal) & $||B_C - B_H||$ & $||B_C-B_P||$ \\
\hline
15 & 0.005s & 0.007s  &0.004s   & 1.3399e-13  & 2.6782e-12   \\
\hline
23 & 0.009s & 0.009s  &0.004s   & 1.0540e-11  &  2.2427e-09   \\
\hline
31 & 0.015s & 0.011s  &  0.005s &  2.3082e-09 &  1.4962e-06  \\
\hline
39 & 0.022s & 0.016s  &  0.006s &   9.7593e-11 &   5.6283e-04  \\
\hline
47 & 0.030s & 0.019s  &  0.006s & 6.6642e-05  &  0.1035 \\
\hline
55 & 0.040s & 0.023s  &  0.007s   & 4.9873e-08  &  457.2366  \\
\hline
63 & 0.053s & 0.027s  & 0.007s  &  1.8852e-05  & 2.2703e+04 \\
\hline
71 & 0.066s & 0.029s  &0.008s   & 6.0574e-07  &  1.7485e+07  \\
\hline
79 & 0.082s & 0.036s  & 0.009s  & 1.0117e-06  & 5.6499e+09 \\
\hline
\end{tabular}
\end{center}
\label{three}
\end{table}

\subsection{Conditioning a Hankel matrix}

It is not rare $n=2m-1$ numbers taken in the interval [0,1] at random result in an ill-conditioned $m\times m$ Hankel matrix $H$. A simple way of handling this is to shift its skew diagonal in order to turn it into a skew-diagonal dominant matrix, $\tilde H=H+\sigma C$, where $C$ is the reciprocal matrix. Let $B_H$ and $B_{\tilde H}$ be the B\'ezier curves corresponding to $H$ and $\tilde H$, respectively. Then, for each $s\in [0,1]$, we compute $B_H(s)$ by subtracting $\sigma$ times $e_m^T B_m^e(s) C_m (B_m^e(s))^Te_m$ from $B_{\tilde H}$. Moreover, this quadratic form has a simple formulation as can be seen in the next lemma.

\begin{lemma} Let $C_m=hankel(e_m,e_1^T)$, which is called the reciprocal matrix. Then, if $w=e^{2\pi i/m}$,
$$e_m^T B_m^e(s) C_m (B_m^e(s))^Te_m=\frac{1}{m} \sum_{j=1}^{m} w^{j-1} (1-s+s.w^{j-1})^{n-1}.$$
\end{lemma}
\begin{proof}
It is easy to see that $C_m = VDV^T$, where $V=vander(1,w,...,w^{m-1})$ and $D = diag(1/m,w/m,...,w^{m-1}/m)$.
From the proof of Corollary \ref{impo},
$$e_m^T B_m^e(s)VDV^T (B_m^e(s))^Te_m =\frac{1}{m} \sum_{j=1}^m w^{j-1} (1-s+s.w^{j-1})^{n-1}.$$
\qquad\end{proof}

\

In table \ref{twoplus}, we can see that this simple technique of preconditioning have improved the computation of B\'ezier curves when their control points yield ill-conditioned Hankel matrices (cond(H) is the maximum condition number of the two Hankel matrices formed by the coordinates of the control points). For each Hankel matrix $H$, $\sigma$ was taken as the sum of the absolute values of its entries. Since our Vandermonde factorization of a Hankel matrix depends on a value chosen at random, the error between the curve computed by Casteljau's and the one computed from that factorization varied enormously when the Hankel matrices associated with the coordinates were ill-conditioned. In table \ref{twoplus}, for each $n$, we can see the maximum error among several experiments done. However, sometimes it happened to have a big error followed by a tiny one.
Notice that all our experiments have been run in a 32-bits AMD Athlon XP 1700+ (1467 MHz).

\begin{table}[ht]
\caption{Mean run time of computation of 129 points of a B\'ezier curve of degree N-1 by three
different methods: Casteljau's (C), Hankel form (H) and preconditioning Hankel form (PH). The results of the second and third methods are compared to the ones obtained by Casteljau's via norm of the difference of the computed points
by the respective method and by Casteljau's.}

\begin{center} \footnotesize
\begin{tabular}{|c|c|c|c|c|c|c|} \hline
N & cond(H) & Time (C) &Time (H) & Time (PH) & $||B_C - B_H||_2$ & $||B_C-B_{PH}||_2$ \\
\hline
31 & 1.5379e+03& 0.015s & 0.012s  &  0.019s &  3.3983e+11 &  2.9510e-11  \\
\hline
39 &760.4605 & 0.022s & 0.016s  &  0.024s &   2.6541e+09 &   1.1134e-10  \\
\hline
47 &2.9956e+03 & 0.030s & 0.019s  &  0.031s &  1.3731e+13 &   1.0189e-10\\
\hline
55 &577.1450 & 0.041s & 0.023s  &0.036s   & 4.3750e+03 & 1.7107e-08   \\
\hline
63 &4.2415e+03 & 0.053s & 0.027s  & 0.042s  &  9.8000e+70  &  2.5894e-08 \\
\hline
71 &907.6247 & 0.066s & 0.029s  &0.049s   & 3.7813e+06  & 3.2318e-07   \\
\hline
79 &1.1167e+03 & 0.081s & 0.033s  & 0.057s  & 4.1314e+04  & 2.1604e-05 \\
\hline
\end{tabular}
\end{center}
\label{twoplus}
\end{table}


\end{document}